\documentclass{amsart}
\pdfoutput=1

\usepackage[T1]{fontenc}                    

\usepackage[english]{babel}                 

\usepackage[usenames,dvipsnames]{color}
\usepackage{graphicx}
\usepackage{tikz}

\usepackage{stmaryrd}
\usepackage{mathtools}									
\usepackage{amsthm}											
\usepackage{amssymb}										
\usepackage{amsmath}
\usepackage{dsfont}											
\usepackage{empheq}											
\usepackage{booktabs}
\usepackage{multirow}
\usepackage{multicol}
\usepackage{verbatim}
\usepackage[ansinew]{inputenc}
\usepackage{enumerate}
\usepackage{caption}

\newcommand{\N}{\ensuremath{\mathbb{N}}}
\newcommand{\Q}{\ensuremath{\mathbb{Q}}}
\newcommand{\Z}{\ensuremath{\mathbb{Z}}}

\newcommand{\R}{\ensuremath{\mathbb{R}}}

\newcommand{\cN}{\ensuremath{\mathcal{N}}}
\newcommand{\cL}{\ensuremath{\mathcal{L}}}

\newcommand{\0}{\ensuremath{\mathbf{0}}}

\DeclareMathOperator{\conv}{conv}
\DeclareMathOperator{\cone}{cone}

\DeclareMathOperator{\ver}{vert}

\theoremstyle{plain}
    \newtheorem*{theo*}{Theorem}
    \newtheorem{theo}{Theorem}
    \newtheorem*{conjecture}{Conjecture}
    \newtheorem{cor}[theo]{Corollary}
    \newtheorem{lemma}[theo]{Lemma}
    
\theoremstyle{definition}
	\newtheorem{defi}[theo]{Definition}
    \newtheorem{example}[theo]{Example}
    
\usepackage{thmtools, thm-restate}

\title{        Convex-normal (pairs of) polytopes}
\author[    Christian Haase \& Jan Hofmann ]{Christian Haase}
\address{Institut f\"ur Mathematik \\ Freie Universit\"at Berlin \\ Germany}
\email{haase@math.fu-berlin.de}
\author[]{Jan Hofmann}
\address{Institut f\"ur Mathematik \\
  Goethe-Universität
  Frankfurt am Main \\ Germany}
\email{jhofmann@math.uni-frankfurt.de}

\begin{document}

\begin{abstract}
In $2012$ Gubeladze (Adv.\ Math.\ 2012) 
introduced the notion of $k$-convex-normal polytopes to show that
integral polytopes all of whose edges are longer than $4d(d+1)$ have
the integer decomposition property.
In the first part of this paper we show that for lattice polytopes
there is no difference between $k$- and $(k+1)$-convex-normality (for
$k\geq 3 $) and improve the bound to $2d(d+1)$. In the second part we
extend the definition to pairs of polytopes and show that for rational
polytopes $P$ and $Q$, where $\cN(P)$ is a refinement of $\cN(Q)$, if
every edge $e_P$ of $P$ is at least $d$ times as long as the
corresponding edge $e_Q$ of $Q$, then $(P+Q)\cap \Z^d = (P\cap \Z^d )
+ (Q \cap \Z^d)$.
\end{abstract}

\maketitle

\section{Introduction }
Polytopes which have the integer decompostion property (IDP) turn up
in many fields of mathematics such as integer programming,
algebraic geometry, where they correspond to projectively normal
embeddings of toric varieties, and in commutative algebra, where
polytopes with the IDP are called integrally closed.\\
So it is natural to ask which polytopes have the IDP. There has
been a lot of research concering this question in recent years. One
way to prove the IDP of a given polytope is to cover it with
simpler polytopes known to have the IDP. The first approach would be
to use the easiest IDP polytopes, namely unimodular simplicies, and
try to show that every polytope with the IDP can be triangulated into
unimodular simplices.
This does not work in general, already in dimension $3$
\cite{KS03}. Relaxing triangulations to coverings with unimodular
simplices, there is a famous $5$-dimensional polytope with the IDP
which does not have such a covering \cite{BG99}. One very nice
positive result is that given a lattice polytope $P$, if all edge
lenghts of $P$ (with respect to the lattice) have a common factor
$c\geq d-1$, then $P$ has the IDP \cite{EW91,LTZ93,BGT97}.\\
The following conjecture proposed during a workshop \cite{HHM07},
suggests that this is also true (maybe with a higher bound) in a more
generalized setting, where the edge-lengths can be independent:
\begin{conjecture}
  Simple lattice polytopes with long edges have the integer
  decompostion property, where long means some invariant, uniform in
  the dimension.
\end{conjecture}
This conjecture was then proved by Gubeladze \cite{Gub12} in the
following precise form.
\begin{theo*}
  Let $P$ be a lattice polytope of dimension $d$. If every edge of $P$
  has lattice length $\geq 4d(d+1)$, then $P$ has the integer
  decompostion property.
\end{theo*}
He proves this theorem in two steps.\\
He first introduces the notion of $k$-convex-normality and proves that
a polytope is $k$-convex-normal if every edge has lattice length
$\geq kd(d+1)$.
Then he shows, that $4$-convex-normal lattice polytopes have the IDP.\\

In the first part of the present paper we further examine
$k$-convex-normal polytopes and show that if $P$ is a lattice polytope
and $k$-convex-normal for some $k\geq 3$, then $P$ is also
$(k+1)$-convex-normal (\textbf{Theorem~\ref{result}}). The lemma used
to prove this theorem, also allows us to improve Gubeladze's bound to
$2d(d+1)$ (\textbf{Corollary~\ref{result2}}).\\

In the second part of the paper we extend the notion of convex-normal
polytopes to pairs of polytopes. We show that given two polytopes $P$
and $Q$, the map $(Q\cap \Z^d) \times (P\cap \Z^d)  \to  (Q+P)\cap
\Z^d$ given by $ (q,p)\mapsto q + p $ is surjective, if the normal fan
of $P$ is a refinement of the normal fan of $Q$ and every edge of $P$
is at least $d$ times as long as its corresponding edge in
$Q$. (\textbf{Theorem~\ref{mainB}})\\

\noindent
\emph{Acknowledgements } We would like to thank Petra Meyer. The first
part of the paper grew out of her master thesis.

\section{Convex-normality revisited}

Let $P \subseteq \R^d$ be a lattice polytope. Then $P$ has the
\emph{integer decompostion property (IDP)}, if for all $k\in \N$ and
all $z \in kP\cap \Z^d$, there exist $x_1,\ldots , x_k \in
P\cap \Z^d$ such that 
\[
	z = x_1 + \cdots + x_k \,.	
\]
Every one or two dimensional lattice polytope has the integer
decompostion property. In dimension $3$ however already simplices do
not need to posses the IDP.\\
For example $P=\conv \{(0,0,0),(1,1,0),(1,0,1),(0,1,1)\}$ does not
have the IDP as $(1,1,1)\in 2P$ is not the sum of two lattice points
in $P$.
\\
Given a rational polytope $Q$ with vertex set $\ver(Q)$ we set
\[
G(Q):=\bigcup_{v\in\ver(Q)} (v+\Z^d)\cap Q \,,
\]
that is, we base the lattice in one vertex after the other and take
the union of those shifted lattices inside $Q$.

Following Gubeladze, we call a rational polytope $P\subseteq \R^d$
\emph{$k$-convex-normal} for some $k\in\Q$, if for all rational $c\in
[2,k]$:
\[
	cP = G((c-1)P) + P \,.
\]
Observe that the inclusion $\supseteq$ is always true.

\begin{example}\label{ConvNormalExample}
In the following picture the polytope $Q=\conv\{
(0,0),(\frac{3}{2},0),(0,\frac{3}{2}) \}$ we get
$G(Q)=\{(0,0),(1,0),(0,1),\quad
(\frac{3}{2},0),(\frac{1}{2},0),(\frac{1}{2},1),\quad
(0,\frac{3}{2}),(0,\frac{1}{2}),(1,\frac{1}{2})\}$.

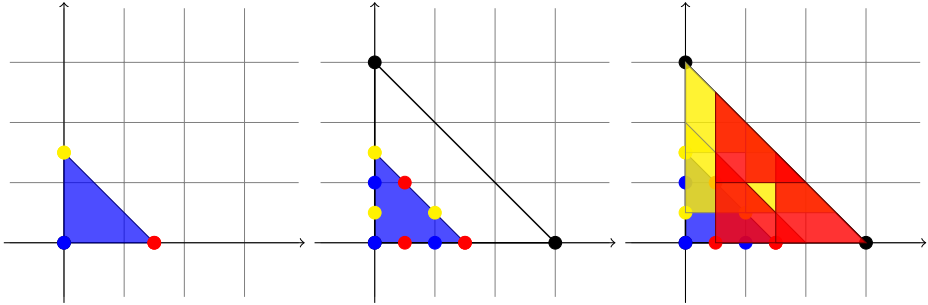
\begin{center}
\begin{tikzpicture}[scale=0.8]
    \draw[step=1cm,gray,very thin] (-0.9,-0.9) grid (3.9,3.9);
    \draw[->] (-1,0) -- (4,0);
    \draw[->] (0,-1) -- (0,4);
                \begin{pgftransparencygroup}
					\pgfsetfillopacity{0.7}
					\filldraw[fill=blue, draw=blue!50!black] (0,0) -- (1.5,0) -- (0,1.5) -- cycle;
				\end{pgftransparencygroup}
   \filldraw   (0,0) circle (3pt);
       \filldraw[color=red]   (1.5,0) circle  (3pt);
       \filldraw[color=yellow]   (0,1.5) circle (3pt);
                \filldraw[color=red]   (1.5,0) circle  (3pt);
                \filldraw[color=yellow]   (0,1.5) circle (3pt);
				\filldraw[color=blue]  (0,0) circle (3pt);
\end{tikzpicture}
\begin{tikzpicture}[scale=0.8]
    \draw[step=1cm,gray,very thin] (-0.9,-0.9) grid (3.9,3.9);
    \draw[->] (-1,0) -- (4,0);
    \draw[->] (0,-1) -- (0,4);
                \begin{pgftransparencygroup}
					\pgfsetfillopacity{0.7}
					\filldraw[fill=blue, draw=blue!50!black] (0,0) -- (1.5,0) -- (0,1.5) -- cycle;
				\end{pgftransparencygroup}
   \filldraw   (0,0) circle (3pt);
       \filldraw[color=red]   (1.5,0) circle  (3pt);
       \filldraw[color=yellow]   (0,1.5) circle (3pt);
                \filldraw[color=black]   (3,0) circle (3pt);
                \filldraw[color=black]   (0,3) circle (3pt);
                \draw[line width=0.2mm]   (0,0) -- (3,0) -- (0,3) -- cycle;
                \filldraw[color=red]   (1.5,0) circle  (3pt);
                \filldraw[color=yellow]   (0,1.5) circle (3pt);
            	\filldraw[color=yellow]  (0,0.5) circle (3pt);
                \filldraw[color=yellow]  (1,0.5) circle (3pt);
		       	\filldraw[color=red]  (0.5,0) circle (3pt);
		       	\filldraw[color=red]  (0.5,1) circle (3pt);
				\filldraw[color=blue]  (0,1) circle (3pt);
				\filldraw[color=blue]  (1,0) circle (3pt);
				\filldraw[color=blue]  (0,0) circle (3pt);
\end{tikzpicture}
\begin{tikzpicture}[scale=0.8]
    \draw[step=1cm,gray,very thin] (-0.9,-0.9) grid (3.9,3.9);
    \draw[->] (-1,0) -- (4,0);
    \draw[->] (0,-1) -- (0,4);
                \begin{pgftransparencygroup}
					\pgfsetfillopacity{0.7}
					\filldraw[fill=blue, draw=blue!50!black] (0,0) -- (1.5,0) -- (0,1.5) -- cycle;
				\end{pgftransparencygroup}
   \filldraw   (0,0) circle (3pt);
       \filldraw[color=red]   (1.5,0) circle  (3pt);
       \filldraw[color=yellow]   (0,1.5) circle (3pt);
                \filldraw[color=black]   (3,0) circle (3pt);
                \filldraw[color=black]   (0,3) circle (3pt);
                \draw[line width=0.2mm]   (0,0) -- (3,0) -- (0,3) -- cycle;
                \filldraw[color=red]   (1.5,0) circle  (3pt);
                \filldraw[color=yellow]   (0,1.5) circle (3pt);
            	\filldraw[color=yellow]  (0,0.5) circle (3pt);
                \filldraw[color=yellow]  (1,0.5) circle (3pt);
		       	\filldraw[color=red]  (0.5,0) circle (3pt);
		       	\filldraw[color=red]  (0.5,1) circle (3pt);
				\filldraw[color=blue]  (0,1) circle (3pt);
				\filldraw[color=blue]  (1,0) circle (3pt);
				\filldraw[color=blue]  (0,0) circle (3pt);
				\begin{pgftransparencygroup}
						\pgfsetfillopacity{0.8}
						\filldraw[fill=yellow, draw=yellow!50!black] (0,1.5) -- (1.5,1.5) -- (0,3) -- cycle;
				\end{pgftransparencygroup}
				\begin{pgftransparencygroup}
						\pgfsetfillopacity{0.8}
						\filldraw[fill=yellow, draw=yellow!50!black] (0,0.5) -- (1.5,0.5) -- (0,2) -- cycle;
				\end{pgftransparencygroup}
				\begin{pgftransparencygroup}
						\pgfsetfillopacity{0.8}
						\filldraw[fill=yellow, draw=yellow!50!black] (1,0.5) -- (2.5,0.5) -- (1,2) -- cycle;
				\end{pgftransparencygroup}
				\begin{pgftransparencygroup}
						\pgfsetfillopacity{0.8}
						\filldraw[fill=red, draw=red!50!black] (0.5,0) -- (2,0) -- (0.5,1.5) -- cycle;
				\end{pgftransparencygroup}
				\begin{pgftransparencygroup}
						\pgfsetfillopacity{0.8}
						\filldraw[fill=red, draw=red!50!black] (1.5,0) -- (3,0) -- (1.5,1.5) -- cycle;
				\end{pgftransparencygroup}
				\begin{pgftransparencygroup}
						\pgfsetfillopacity{0.8}
						\filldraw[fill=red, draw=red!50!black] (0.5,1) -- (0.5,2.5) -- (2,1) -- cycle;
				\end{pgftransparencygroup}				
\end{tikzpicture}
\end{center}
The colors in the figure encode which vertex produced the base point
for the corresponding copy of $P$ and we can see that $Q$ is
$2$-convex-normal.\\
An easy example of a polytope which is not convex normal is the
$2$-dimensional standard simplex $Q=\conv \{(0,0),(1,0),(0,1)\}$:
\begin{center}
\begin{tikzpicture}[scale=1]
    \draw[step=1cm,gray,very thin] (-0.9,-0.9) grid (2.9,2.9);
    \draw[->] (-1,0) -- (3,0);
    \draw[->] (0,-1) -- (0,3);
                \begin{pgftransparencygroup}
					\pgfsetfillopacity{0.7}
					\filldraw[fill=blue, draw=blue!50!black] (0,0) -- (1,0) -- (0,1) -- cycle;
				\end{pgftransparencygroup}
				                \draw[line width=0.2mm]   (0,0) -- (2,0) -- (0,2) -- cycle;
   				\filldraw[color=blue]   (0,0) circle (3pt);
       			\filldraw[color=red]   (1,0) circle  (3pt);
       			\filldraw[color=yellow]   (0,1) circle (3pt);
                \filldraw[color=black]   (2,0) circle (3pt);
                \filldraw[color=black]   (0,2) circle (3pt);
                \filldraw[color=red]   (1,0) circle  (3pt);
\end{tikzpicture}
\begin{tikzpicture}[scale=1]
    \draw[step=1cm,gray,very thin] (-0.9,-0.9) grid (2.9,2.9);
    \draw[->] (-1,0) -- (3,0);
    \draw[->] (0,-1) -- (0,3);
                \begin{pgftransparencygroup}
					\pgfsetfillopacity{0.7}
					\filldraw[fill=blue, draw=blue!50!black] (0,0) -- (1,0) -- (0,1) -- cycle;
				\end{pgftransparencygroup}
				                \draw[line width=0.2mm]   (0,0) -- (2,0) -- (0,2) -- cycle;
                \filldraw[color=red]   (1,0) circle  (3pt);
                \begin{pgftransparencygroup}
					\pgfsetfillopacity{0.8}
                	\filldraw[fill=red, draw=red!50!black] (1,0) -- (2,0) -- (1,1) -- cycle;
				\end{pgftransparencygroup}
				\begin{pgftransparencygroup}
					\pgfsetfillopacity{0.8}
                	\filldraw[fill=yellow, draw=yellow!50!black] (0,1) -- (0,2) -- (1,1) -- cycle;
				\end{pgftransparencygroup}
   				\filldraw[color=blue]   (0,0) circle (3pt);
       			\filldraw[color=red]   (1,0) circle  (3pt);
       			\filldraw[color=yellow]   (0,1) circle (3pt);
                \filldraw[color=black]   (2,0) circle (3pt);
                \filldraw[color=black]   (0,2) circle (3pt);
\end{tikzpicture}
\end{center}
\end{example}
Our first lemma highlights a special behavior of $G(rP)$, when $P$ is
a lattice polytope.

\begin{lemma}\label{A}
  Let $P$ be a lattice polytope and $r\in \R_{>0}$, then 
  \[ G(rP)+G(P) \subseteq  G((r+1)P). \]
\end{lemma}
\begin{proof}
  Let $x = rv + u \in G(rP)$ and $y= w + u'\in G(P)$ with $v,w\in
  \ver(P)$ and $u,u',v,w\in \Z^d$. As $x\in rP$ and $y\in P$ it
  follows that $z=x+y\in (r+1)P$ and also
  \[
  z = x + y = rv + u + w + u' = (r+1)v + ( w - v + u + u') \in
  \ver((r+1)P)+\Z^d
  \]
  so $z\in G((r+1)P)$.
\end{proof}
The previous lemma yields the induction step for our first theorem.
\begin{lemma}\label{Main lemma}
	Let $P$ be a $2$-convex-normal lattice polytope, then
	\[ G((c-2)P) + P=(c-1)P \text{ implies } G((c-1)P) + P = cP. \]
\end{lemma}
\begin{proof}
  $G((c-1)P) + P\subseteq cP$ is always true, hence we only have to
  show the other direction $cP \subseteq G((c-1)P) + P$: 
  \[
  cP = (c-1)P + P = \left( G((c-2)P) + P\right) + P  = G((c-2)P) + 2P
  \]
  but $P$ is $2$-convex normal so that $2P = G(P) + P$ and hence:
  \[
  cP = G((c-2)P) + 2P = G((c-2)P) + G(P) + P \subseteq G((c-1)P) + P
  \]
  where the inclusion follows from Lemma~\ref{A}.
\end{proof}
This lemma has two very nice consequences.
\begin{theo}\label{result}
  Let $P$ be a lattice polytope. If $P$ is $3$-convex-normal, then $P$
  is also $k$-convex-normal, for all $k\geq 2$.
\end{theo}
Let $e$ be the edge-vector of a rational polytope $P$ connecting
vertices $v$ and $w$, such that $v + e = w$. By $\ell(e)$ we denote
the lattice length of $e$, i.e. let $u$ be the smallest integer vector
on the line spanned by $w-v$ then $e=ku$ for some $k\in \Q$ and
$\ell(e):=\vert k \vert$.
The previous Theorem together with  \cite[Lemma $6.2$]{Gub12} implies
that a lower bound of $\ell(e) \geq 3d(d+1)$ for every edge $e$ of $P$
would be enough. But using Lemma~\ref{Main lemma} directly, we can do
better.
\begin{cor}
  Let $P$ be a lattice polytope. If $P$ is 2-convex-normal, then $P$
  has the integer decompositions property.
\end{cor}
\begin{proof}
  As $P$ is $2$-convex-normal, by Lemma~\ref{Main lemma} we have that
  $kP = G((k-1)P) + P$ for all $k \in \N$.\\ Now given $z \in kP\cap
  \Z^d$ for some $k\in \N$, we know that $z = x + y$ where $x\in
  G((k-1)P)=(k-1)P\cap \Z^d$ and therefore $y\in P\cap\Z^d$. By
  induction we can find $x_1,\ldots, x_{k-1}\in P\cap\Z^d$ such that
  $x = x_1 + \ldots + x_{k-1}$.
\end{proof}
Now, combine the last corollary with \cite[Theorem $1.2$]{Gub12}.
\begin{cor}\label{result2}
  Let $P$ be a lattice polytope. If for every edge $e$ of $P$ the
  lattice length $ \ell(e)\geq 2d(d+1)$, then $P$ has the integer
  decompostion property.
\end{cor}

\section{Convex-normality for pairs of polytopes}
In this chapter we extend the above definitions and results to pairs
of polytopes.
\begin{defi}
A pair of (rational) polytopes $(Q,P)$ is called \emph{ convex -
  normal }, if
\[
Q + P = G(Q) + P 
\]
\end{defi}

Note, that we only have to show $Q+P \subseteq G(Q)+P$ as the other
direction is always true since $G(Q)\subset Q$. Furthermore this
notion is translation invariant, as a small calculation shows that
$G(Q-w) = G(Q) -w$. Hence we can set two vertices $v\in \ver (P)$ and
$w\in \ver(Q)$ to $\0$.
\begin{example}
As seen in \textbf{Example \ref{ConvNormalExample}} the pair
$(1.5\cdot\Delta_2,1.5\cdot\Delta_2)$ is convex-normal and the pair
$(\Delta_2,\Delta_2)$ is not. More generally, $P$ is $2$-convex-normal
if and only if $(P,P)$ is convex-normal.
\end{example}
\begin{example}
Convex-normality is not symmetric. When we set
\[
P=\conv \begin{pmatrix}
0 & 1 & 0 & 1 \\
0 & 0 & 1 & 1
\end{pmatrix}\quad \text{ and } \quad Q=\conv\begin{pmatrix}
0 & 1 & 0 & 1 \\
0 & 0 & 0.7 & 0.7
\end{pmatrix}
\]
the following pictures illustrate that $G(Q) + P = Q + P$ but $G(P) + Q \not = P + Q$:\\
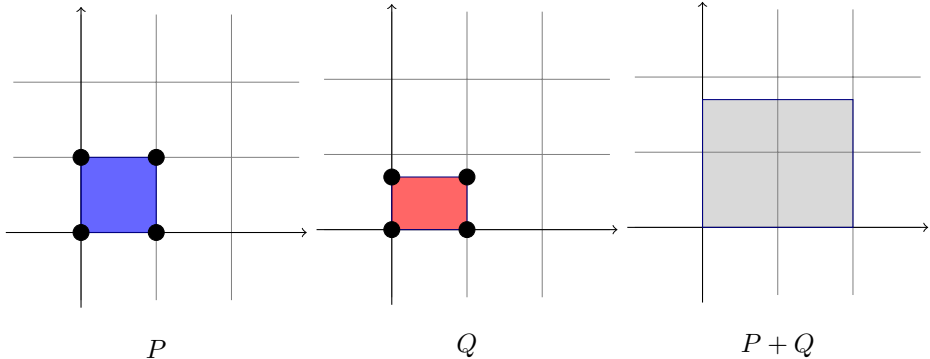
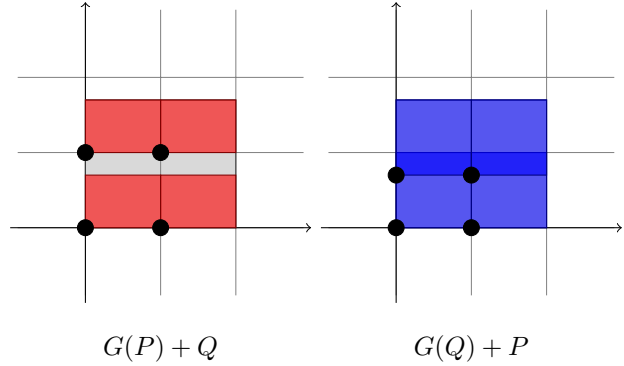
\begin{center}
\begin{tikzpicture}[scale=1]
    \draw[step=1cm,gray,very thin] (-0.9,-0.9) grid (2.9,2.9);
    \draw[->] (-1,0) -- (3,0);
    \draw[->] (0,-1) -- (0,3);
                \begin{pgftransparencygroup}
					\pgfsetfillopacity{0.6}
					\filldraw[fill=blue, draw=blue!50!black] (0,0) -- (1,0) -- (1,1) -- (0,1) -- cycle;
				\end{pgftransparencygroup}
   		\filldraw   (0,0) circle (3pt);
   		\filldraw	(1,0) circle (3pt)
   		   		node[below=1.3cm]{$P$};
   		\filldraw	(0,1) circle (3pt);
   		\filldraw	(1,1) circle (3pt);
\end{tikzpicture}
\begin{tikzpicture}[scale=1]
    \draw[step=1cm,gray,very thin] (-0.9,-0.9) grid (2.9,2.9);
    \draw[->] (-1,0) -- (3,0);
    \draw[->] (0,-1) -- (0,3);
                \begin{pgftransparencygroup}
					\pgfsetfillopacity{0.6}
					\filldraw[fill=red, draw=blue!50!black] (0,0) -- (1,0) -- (1,0.7) -- (0,0.7) -- cycle;
				\end{pgftransparencygroup}
   		\filldraw   (0,0) circle (3pt);
   		\filldraw	(1,0) circle (3pt)
   				node[below=1.27cm]{$Q$};
   		\filldraw	(0,0.7) circle (3pt);
   		\filldraw	(1,0.7) circle (3pt);
\end{tikzpicture}
\begin{tikzpicture}[scale=1]
    \draw[step=1cm,gray,very thin] (-0.9,-0.9) grid (2.9,2.9);
    \draw[->] (-1,0) -- (3,0);
    \draw[->] (0,-1) -- (0,3);
      		\begin{pgftransparencygroup}
      			\pgfsetfillopacity{0.3}
      			\filldraw[fill=gray, draw=blue!50!black] (0,0) -- (2,0) -- (2,1.7) -- (0,1.7) -- cycle;
      		\end{pgftransparencygroup}
      	\filldraw	(1,0) circle (0pt)
      			node[below=1.3cm]{$P+Q$};	
\end{tikzpicture}

\begin{tikzpicture}[scale=1]
    \draw[step=1cm,gray,very thin] (-0.9,-0.9) grid (2.9,2.9);
    \draw[->] (-1,0) -- (3,0);
    \draw[->] (0,-1) -- (0,3);
      		\begin{pgftransparencygroup}
      			\pgfsetfillopacity{0.3}
      			\filldraw[fill=gray, draw=gray!50!black] (0,0) -- (2,0) -- (2,1.7) -- (0,1.7) -- cycle;
      		\end{pgftransparencygroup}	
      		\begin{pgftransparencygroup}
      			\pgfsetfillopacity{0.6}
      			\filldraw[fill=red, draw=red!50!black] (0,0) -- (1,0) -- (1,0.7) -- (0,0.7) -- cycle;
      		\end{pgftransparencygroup}
      		\begin{pgftransparencygroup}
      			\pgfsetfillopacity{0.6}
      			\filldraw[fill=red, draw=red!50!black] (1,0) -- (2,0) -- (2,0.7) -- (1,0.7) -- cycle;
      		\end{pgftransparencygroup}		
      		\begin{pgftransparencygroup}
      			\pgfsetfillopacity{0.6}
      			\filldraw[fill=red, draw=red!50!black] (0,1) -- (1,1) -- (1,1.7) -- (0,1.7) -- cycle;
      		\end{pgftransparencygroup}	
      		\begin{pgftransparencygroup}
      			\pgfsetfillopacity{0.6}
      			\filldraw[fill=red, draw=red!50!black] (1,1) -- (2,1) -- (2,1.7) -- (1,1.7) -- cycle;
      		\end{pgftransparencygroup}	
      		      		\filldraw   (0,0) circle (3pt);
      		\filldraw	(1,0) circle (3pt)
      				node[below=1.3cm]{$G(P)+Q$};
      		\filldraw	(0,1) circle (3pt);
      		\filldraw	(1,1) circle (3pt);
\end{tikzpicture}
\begin{tikzpicture}[scale=1]
    \draw[step=1cm,gray,very thin] (-0.9,-0.9) grid (2.9,2.9);
    \draw[->] (-1,0) -- (3,0);
    \draw[->] (0,-1) -- (0,3);
      		\begin{pgftransparencygroup}
      			\pgfsetfillopacity{0.3}
      			\filldraw[fill=gray, draw=blue!50!black] (0,0) -- (2,0) -- (2,1.7) -- (0,1.7) -- cycle;
      		\end{pgftransparencygroup}	
      		\begin{pgftransparencygroup}
      			\pgfsetfillopacity{0.6}
      			\filldraw[fill=blue, draw=blue!50!black] (0,0) -- (1,0) -- (1,1) -- (0,1) -- cycle;
      		\end{pgftransparencygroup}
      		\begin{pgftransparencygroup}
      			\pgfsetfillopacity{0.6}
      			\filldraw[fill=blue, draw=blue!50!black] (1,0) -- (2,0) -- (2,1) -- (1,1) -- cycle;
      		\end{pgftransparencygroup}		
      		\begin{pgftransparencygroup}
      			\pgfsetfillopacity{0.6}
      			\filldraw[fill=blue, draw=blue!50!black] (0,0.7) -- (1,0.7) -- (1,1.7) -- (0,1.7) -- cycle;
      		\end{pgftransparencygroup}	
      		\begin{pgftransparencygroup}
      			\pgfsetfillopacity{0.6}
      			\filldraw[fill=blue, draw=blue!50!black] (1,0.7) -- (2,0.7) -- (2,1.7) -- (1,1.7) -- cycle;
      		\end{pgftransparencygroup}	
      		\filldraw   (0,0) circle (3pt);
      		\filldraw	(1,0) circle (3pt)
      				node[below=1.3cm]{$G(Q)+P$};
      		\filldraw	(0,0.7) circle (3pt);
      		\filldraw	(1,0.7) circle (3pt);      		
\end{tikzpicture}
\end{center}
\end{example}

The second definition we need is an extension of the integer
decomposition property to pairs of polytopes:

\begin{defi}
A pair of lattice polytopes $(Q,P)$ has the \emph{integer
  decomposition property (IDP)}, if the map
\[
\begin{array}{ccccc}
  (Q\cap \Z^d) &\times& (P\cap \Z^d) & \to & (Q+P)\cap \Z^d  \\
  (q&,&p)&\mapsto& q + p \qquad
\end{array}
\]
is surjective, that is, if $(P+Q)\cap \Z^d =
(P\cap\Z^d)+(Q\cap\Z^d)$.
\end{defi}

If the pairs $(P,nP)$ have the integer decomposition property for all
$n \in \N$, then $P$ has it, too.\\
The pair $(\Delta_2,\Delta_2)$ from the example above has the integer
decomposition property, so we see that pairs of polytopes with the IDP
are not always convex-normal. But the converse implication is true:

\begin{lemma}
  Let $P$ be a rational polytope and let $Q$ be a lattice polytope
  such that $(Q,P)$ is convex-normal. Then $(Q,P)$ has the integer
  decomposition property.
\end{lemma}
\begin{proof}
As $(Q,P)$ is convex-normal, we know that $Q+P = G(Q) + P$.\\
As $Q$ is a lattice polytope, we have  $G(Q)= Q\cap \Z^d$ and hence
\[
(Q+P)\cap \Z^d = (G(Q)+P)\cap \Z^d = ((Q\cap \Z^d)+P)\cap \Z^d =
(Q\cap \Z^d) + (P\cap \Z^d)
\]
\end{proof}

In the remainder of this paper we will prove a sufficient condition,
based on edge lengths, for a pair $(Q,P)$ to be convex-normal.\\
Given a polytope $P$, if $F$ is a face of $P$ we write $F\prec P$. For
every nonempty face $F$ of $P$ there exists a linear functional $c_F$,
such that $c_F^tx$ is maximal over $P$ iff $x\in F$. We also say that
$c_F$ defines the face $F$.
The set 
\[
C_F = \left\{c\, : \;\; \{z : \max_{x\in P}c^tx = c^tz \} \supseteq F
\right\}
\] 
is a polyhedral cone.\\
The \emph{normal fan} $\cN(P)$ of $P$ is the collection of these cones
over all nonempty faces of $P$. The correspondence $F
\longleftrightarrow C_F$ is an inclusion reversing bijection. I.e.,
given two faces $F,F'\prec P$, then $F\subseteq F'$ if and only if
$C_{F'}\subseteq C_F$.\\
In the above examples $P$ and $Q$ had the same normal fan. If we drop
this condition, there are pairs of polytopes with arbitrarily long
edges lacking the integer decomposition property and not being
convex-normal.
\begin{example}
Set
\[
Q=\conv \begin{pmatrix}
0 & 1 & 0 \\
0 & k & 1
\end{pmatrix}\quad \text{ and } \quad P=\conv\begin{pmatrix}
0 & -l & -(l-1) \\
0 & 1 & 1 
\end{pmatrix}
\]
\begin{figure}[h]
\begin{center}
\begin{tikzpicture}[scale=1]
    \draw[step=1cm,gray,very thin] (-3.9,-0.9) grid (2.9,2.9);
    \draw[->] (-4,0) -- (3,0);
    \draw[->] (0,-1) -- (0,3);
                \begin{pgftransparencygroup}
					\pgfsetfillopacity{0.6}
					\filldraw[fill=red, draw=blue!50!black] (0,0) -- (1,2) --  (0,1) -- cycle;
				\end{pgftransparencygroup}
				\begin{pgftransparencygroup}
					\pgfsetfillopacity{0.6}
					\filldraw[fill=red, draw=blue!50!black] (0,0) -- (-2,1) --  (-3,1) -- cycle;
				\end{pgftransparencygroup}
\end{tikzpicture}

\begin{tikzpicture}[scale=1]
    \draw[step=1cm,gray,very thin] (-3.9,-0.9) grid (2.9,2.9);
    \draw[->] (-4,0) -- (3,0);
    \draw[->] (0,-1) -- (0,3);
                \begin{pgftransparencygroup}
					\pgfsetfillopacity{0.6}
					\filldraw[fill=gray, draw=gray!50!black] (0,0) -- (-3,1) --  (-3,2) -- (-2,3) -- (-1,3) -- (1,2) -- cycle;
				\end{pgftransparencygroup}
				\filldraw	(0,0) circle (3pt)
							(0,1) circle (3pt)
							(1,2) circle (3pt);
				\begin{pgftransparencygroup}
					\pgfsetfillopacity{0.6}
					\filldraw[fill=red, draw=blue!50!black] (0,0) -- (-3,1) --  (-2,1) -- cycle;
				\end{pgftransparencygroup}
				\begin{pgftransparencygroup}
					\pgfsetfillopacity{0.6}
					\filldraw[fill=red, draw=blue!50!black] (0,1) -- (-3,2) --  (-2,2) -- cycle;
				\end{pgftransparencygroup}
				\begin{pgftransparencygroup}
					\pgfsetfillopacity{0.6}
					\filldraw[fill=red, draw=blue!50!black] (1,2) -- (-2,3) --  (-1,3) -- cycle;
				\end{pgftransparencygroup}
				\draw[color=blue]	(-1,1) circle (3pt)
									(-1,2) circle (3pt)
									(0,2)  circle (3pt);
\end{tikzpicture}
\end{center}
\caption*{$Q+P$ and $G(Q)+P$ for $n=1$, $k=2$ and $l=3$.}
\end{figure}
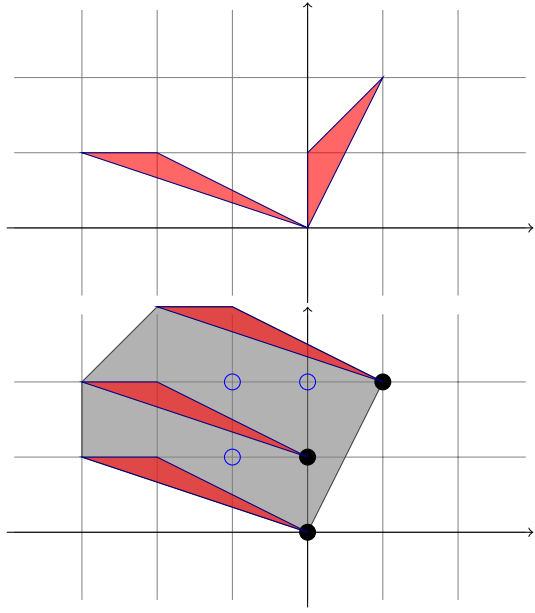

If we look at $(nQ,nP)$, then both polytopes have edge length $n$ and
there are $O(n^4)$ lattice points in $(nP\cap \Z^2)+(nQ\cap \Z^2)$,
but $k\cdot l \cdot O(n^2)$ lattice points in $nP + nQ$. Hence for
$k,l \gg n$, the pair $(nQ,nP)$ neither has the integer decomposition
property nor is it convex-normal.
\end{example}
For a pair $(Q,P)$ of polytopes to be convex-normal, it is not
enough if both polytopes have the integer decomposition property, be
$k$-convex-normal or have long edges and the examples suggest that we
need a condition on the normal fans of $P$ and $Q$ and in fact that is
what we need.\\
Given two $d$-polytopes $Q$ and $P$, if $\cN(P)$ is a refinement of
$\cN(Q)$ (as in finer subdivision of $\R^d$), then for every cone
$C\in \cN(P)$ there exists a cone $D\in \cN(Q)$ s.t. $C\subseteq D$.
In this case we can define a map $\Phi' : \cN(P) \to \cN(Q)$
s.t. $\Phi'(C)$ is defined as the smallest cone in $\cN(Q)$ containing
$C$. This map preserves inclusions and has a corresponding map $\Phi:
\cL(P) \to \cL(Q)$ on the face lattices of $P$ and $Q$, taking a face
$F\prec P$ with corresponding cone $C_F$ to the face $G\prec Q$ with
corresponding cone $C_G=\Phi'(C_F)$.\\
\begin{example}
We illustrate the map in the following picture with 
\[
P = \conv \begin{pmatrix}
0 & 3 & 3 & 2 & -1 & -1 \\
0 & 0 & -2 & -3 & -3 & -1 \\
\end{pmatrix}
\quad \text{and} \quad Q = \conv \begin{pmatrix}
0 & 2 & 2 & 0 \\
0 & 0 & -2 & -2 \\
\end{pmatrix}
\]
\begin{center}
\begin{tikzpicture}[scale=0.8]
    \draw[step=1cm,gray,very thin] (-1.9,-3.9) grid (3.9,0.9);
    \draw[->] (-2,0) -- (4,0);
    \draw[->] (0,1) -- (0,-4);
    \begin{pgftransparencygroup}
		\pgfsetfillopacity{0.6}
		\filldraw[fill=gray] (0,0) -- (-1,-1) -- (-1,-3) -- (2,-3) -- (3,-2) -- (3,0) -- cycle;
	\end{pgftransparencygroup}
	\filldraw[fill=red] (0,0) circle (3pt);
	\filldraw[fill=yellow] (3,0) circle (3pt);
	\filldraw[fill=green] (3,-2) circle (3pt);
	\filldraw[fill=green!50!black] (2,-3) circle (3pt);
	\filldraw[fill=blue] (-1,-3) circle (3pt)
		node[below=1cm, right=0.7cm]{P \& \cN(P)};
	\filldraw[fill=red!50!black] (-1,-1) circle (3pt)
		node[above=0.5cm, right=0.1cm]{e};
\end{tikzpicture}
\qquad
\begin{tikzpicture}[scale=0.8]
    \draw[step=1cm,gray,very thin] (-0.9,-2.9) grid (2.9,0.9);
    \draw[->] (-1,0) -- (3,0);
    \draw[->] (0,1) -- (0,-3);
                \begin{pgftransparencygroup}
					\pgfsetfillopacity{0.6}
					\filldraw[fill=gray] (0,0) -- (0,-2) --  (2,-2) -- (2,0) -- cycle;
				\end{pgftransparencygroup}
	\filldraw[fill=red] (0,0) circle (3pt);
	\filldraw[fill=yellow] (2,0) circle (3pt);
	\filldraw[fill=green] (2,-2) circle (3pt);
	\filldraw[fill=blue] (0,-2) circle (3pt)
		node[below=1cm , right=0cm]{Q \& \cN(Q)};

\end{tikzpicture}
\end{center}

\begin{center}
\begin{tikzpicture}[scale=0.8]
	\begin{pgftransparencygroup}
		\pgfsetfillopacity{0.6}
		\fill[fill=red!50!black] (0,0) -- (-2,0) -- (-2,2) -- cycle;
		\fill[fill=red] (0,0) -- (0,2) -- (-2,2) -- cycle;
	    \fill[fill=yellow] (0,2) -- (2,0) -- (0,0) -- cycle;
	    \fill[fill=green] (0,0) -- (2,0) -- (2,-2) -- cycle;
	    \fill[fill=green!50!black] (0,0) -- (0,-2) -- (2,-2) -- cycle;
	    \fill[fill=blue] (0,0) -- (0,-2) -- (-2,0) -- cycle;
	\end{pgftransparencygroup}
    \draw[->,ultra thick] (0,0) -- (0,2.05);
    \draw[->,ultra thick] (0,0) -- (2.05,0);
    \draw[->,ultra thick] (0,0) -- (2.05,-2.05);
    \draw[->,ultra thick] (0,0) -- (0,-2.05);
    \draw[->,ultra thick] (0,0) -- (-2.05,0);
    \draw[->,ultra thick] (0,0) -- (-2.05,2.05);
\end{tikzpicture}
\qquad \qquad
\begin{tikzpicture}[scale=0.8]
	\begin{pgftransparencygroup}
		\pgfsetfillopacity{0.6}
	    \fill[fill=red] (0,0) -- (0,2) -- (-2,0)-- cycle;
	    \fill[fill=yellow] (0,0) -- (0,2) -- (2,0) -- cycle;
	    \fill[fill=green] (0,0) -- (2,0) -- (0,-2) -- cycle;
	    \fill[fill=blue] (0,0) -- (0,-2) -- (-2,0) -- cycle;
	\end{pgftransparencygroup}
    \draw[->,ultra thick] (0,0) -- (0,2);
    \draw[->,ultra thick] (0,0) -- (2,0);
    \draw[->,ultra thick] (0,0) -- (0,-2);
    \draw[->,ultra thick] (0,0) -- (-2,0);
\end{tikzpicture}
\end{center}
For example the edge $e$ from $(-1,-1)$ to $(0,0)$ in $P$ corresponds
to the vertex $(0,0)$ in $Q$, i.e. $\Phi(e)=(0,0)$ because $e$
corresponds to $\cone\begin{pmatrix}
  -1\\
  1
\end{pmatrix} \in \cN(P)$ and the smallest cone of $\cN(Q)$ containing
it is $\cone \begin{pmatrix}
  -1 & 0 \\
  0  & 1
\end{pmatrix}$, which is the normal cone belonging to $(0,0)$ in $Q$.
\end{example}

\section{A sufficient criterion for convex-normality of $(Q,P)$}
Now that we got all the tools lined up, we can start the proof with
the following lemma, which is the base case for our induction:
\begin{lemma}\label{basecase}
  Let $P=[0,q]$ and $Q=[0,m]$ be intervals with $q\geq \min \{1 , m
  \}$, then $(Q,P)$ is convex-normal.
\end{lemma}

\begin{proof}
Set $l:=\lfloor m \rfloor$. If $l\geq 1$, then 
\[
Q+P = [0,q+m] =\left(\bigcup_{i=0}^l i + [0,q]\right) \cup m + [0,q]
\subseteq G(Q) + P
\]
as $q \geq 1$.\\
If $l<1$, then:
\[
Q+P = ( 0 + P ) \cup (m + P)
\]
as $q\geq l$.
\end{proof}

Now we can prove the main result.
\begin{theo}\label{mainB}
  Let $P$ and $Q$ be rational $d$-polytopes such that $\cN(P)$ is a
  refinement of $\cN(Q)$ and such that $\ell(e_P)\geq d\cdot
  \ell(e_Q)$ for all edges $e_P\prec P$ and $e_Q\prec Q$, where
  $e_Q=\Phi(e_P)$.
  Then $(Q,P)$ is convex-normal.
\end{theo}

\begin{proof}
Lemma \ref{basecase} took care of the base case, hence let $P$ and $Q$
be $d$-polytopes with $d\geq 2$.
\text{ }\\

\noindent
\textbf{STEP 1 - SUBDIVIDING $Q+P$:}\\
We start by subdividing $Q + P$ by assigning weights/heights to the
vertices of $P$ and $Q$, where without loss of generality $\0\in
\ver(P)$ and $\0=\Phi(0)\in \ver (Q)$. Vertices of $Q$ and the vertex
$\0$ of $P$ get height $0$ and all the other vertices of $P$ get
height $1$. We use those heights to define new polytopes $P'$ and $Q'$
in $\R^{d+1}$ as follows.
\[
Q':=\conv \{(w,0) : w \in \ver(Q)\} \text{ and } P':=\conv\left(
  (\0,0)\cup\left\{(u,1) :u\in \ver(P)\backslash\{v\}\right\}\right)
\]
Then the projection of $P'+Q'$ onto the first $d$ coordinates is $P+Q$
and the lower hull of $P'+Q'$ induces a subdivision of $P+Q$ into the
following pieces.\\
\[
0 + Q \text{ and } F_Q+(\conv(0,F_P))
\]
$\text{ for faces } F_Q\prec Q \text{ and faces } F_P \prec P, \text{
  with } 0\not\in F_P \text{ and } \phi(F_P)=F_Q.$

\begin{center}
\begin{tikzpicture}[scale=0.6]
    \draw[step=1cm,gray,very thin] (-2.9,-6.9) grid (6.9,0.9);
    \draw[->] (-3,0) -- (7,0);
    \draw[->] (0,1) -- (0,-7);
                \begin{pgftransparencygroup}
					\pgfsetfillopacity{0.6}
					\filldraw[fill=green, draw=green!50!black] (0,0) -- (-2,-2) --  (-2,-6) -- (0,0) -- cycle;
				\end{pgftransparencygroup}
                \begin{pgftransparencygroup}
					\pgfsetfillopacity{0.6}
					\filldraw[fill=blue, draw=blue!50!black] (-2,-6) -- (4,-6) --  (0,0) -- cycle;
				\end{pgftransparencygroup}
                \begin{pgftransparencygroup}
					\pgfsetfillopacity{0.6}
					\filldraw[fill=yellow, draw=red!50!black] (4,-6) -- (6,-4) -- (0,0) -- cycle;
				\end{pgftransparencygroup}										
                \begin{pgftransparencygroup}
					\pgfsetfillopacity{0.6}
					\filldraw[fill=red, draw=blue!50!black] (0,0) -- (6,-4) -- (6,0) -- cycle;
				\end{pgftransparencygroup}
	\draw (0,0) -- (-2,-6);
	\draw (0,0) -- (4,-6);
	\draw (0,0) -- (6,-4);
	\filldraw (0,0) circle (3pt);
	\draw[fill=blue] (2,-6) circle (0pt)
	node[below=0.7cm]{$P$};
\end{tikzpicture}
\begin{tikzpicture}[scale=0.6]
    \draw[step=1cm,gray,very thin] (-0.9,-2.9) grid (2.9,0.9);
    \draw[->] (-1,0) -- (3,0);
    \draw[->] (0,1) -- (0,-3);
                \begin{pgftransparencygroup}
					\pgfsetfillopacity{0.6}
					\filldraw[fill=gray] (0,0) -- (0,-2) --  (2,-2) -- (2,0) -- cycle;
				\end{pgftransparencygroup}
	\filldraw (0,0) circle (3pt);
	\draw[fill=blue] (1,-2) circle (0pt)
	node[below=0.7cm]{$Q$};
\end{tikzpicture}
\end{center}

\begin{center}
\begin{tikzpicture}[scale=0.6]
    \draw[step=1cm,gray,very thin] (-2.9,-8.9) grid (8.9,0.9);
    \draw[->] (-3,0) -- (9,0);
    \draw[->] (0,1) -- (0,-9);
                \begin{pgftransparencygroup}
					\pgfsetfillopacity{0.6}
					\filldraw[fill=gray, draw=gray!50!black] (0,0) -- (-2,-2) --  (-2,-6) -- (4,-6) -- (6,-4) -- (6,0) -- cycle;
				\end{pgftransparencygroup}									
                \begin{pgftransparencygroup}
					\pgfsetfillopacity{0.6}
					\filldraw[fill=gray, draw=gray!50!black] (0,0) -- (0,-2) --  (2,-2) -- (2,0) -- cycle;
				\end{pgftransparencygroup}
                \begin{pgftransparencygroup}
					\pgfsetfillopacity{0.6}
					\filldraw[fill=green, draw=green!50!black] (0,0) -- (-2,-2) --  (-2,-8) -- (0,-2) -- cycle;
				\end{pgftransparencygroup}
                \begin{pgftransparencygroup}
					\pgfsetfillopacity{0.6}
					\filldraw[fill=blue, draw=blue!50!black] (-2,-8) -- (6,-8) -- (2,-2) -- (0,-2) -- cycle;
				\end{pgftransparencygroup}
                \begin{pgftransparencygroup}
					\pgfsetfillopacity{0.6}
					\filldraw[fill=yellow, draw=yellow!50!black] (6,-8) -- (8,-6) -- (2,-2) -- cycle;
				\end{pgftransparencygroup}	
                \begin{pgftransparencygroup}
					\pgfsetfillopacity{0.6}
					\filldraw[fill=red, draw=red!50!black] (8,-6) -- (8,0) -- (2,0) -- (2,-2) -- cycle;
				\end{pgftransparencygroup}
	\draw (0,-2) -- (-2,-8);
	\draw (2,-2) -- (6,-8);
	\draw (2,-2) -- (8,-6);
	\filldraw (0,0) circle (3pt);
	\draw[fill=blue] (3,-8) circle (0pt)
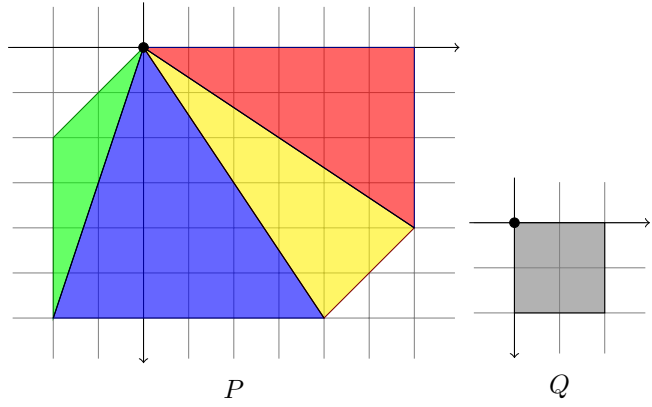
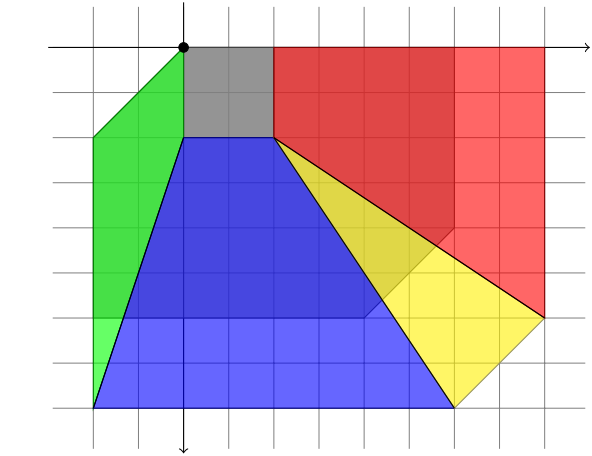
	node[below=0.7cm]{$Q+P$ subdivided into $0 + Q \text{ and } F_Q+(\conv(0,F_P))$};
\end{tikzpicture}
\end{center}

Another decomposition of $P + Q$ we will be using, is the following:
\[
I:= \left(\frac{d-1}{d}\right)P + Q \quad \text{ and } \quad B:=
\overline{\left( P+Q \right) - I}
\]
Where $I$ stands for the ``inner'' part of $P+Q$ and $B$ stands for
the ``boundary'' part of $P+Q$.
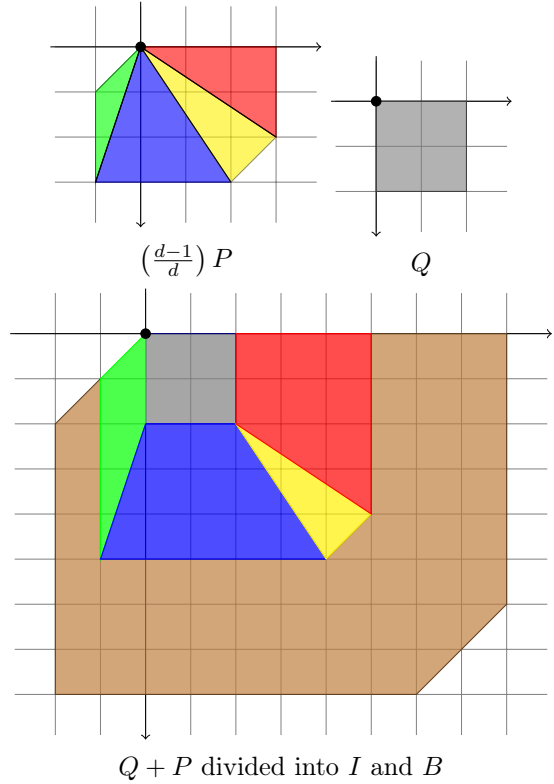
\begin{center}
\begin{tikzpicture}[scale=0.6]
    \draw[step=1cm,gray,very thin] (-1.9,-3.9) grid (3.9,0.9);
    \draw[->] (-2,0) -- (4,0);
    \draw[->] (0,1) -- (0,-4);
                \begin{pgftransparencygroup}
					\pgfsetfillopacity{0.6}
					\filldraw[fill=green, draw=green!50!black] (0,0) -- (-1,-1) --  (-1,-3) -- (0,0) -- cycle;
				\end{pgftransparencygroup}
                \begin{pgftransparencygroup}
					\pgfsetfillopacity{0.6}
					\filldraw[fill=blue, draw=blue!50!black] (-1,-3) -- (2,-3) --  (0,0) -- cycle;
				\end{pgftransparencygroup}
                \begin{pgftransparencygroup}
					\pgfsetfillopacity{0.6}
					\filldraw[fill=yellow, draw=yellow!50!black] (2,-3) -- (3,-2) -- (0,0) -- cycle;
				\end{pgftransparencygroup}										
                \begin{pgftransparencygroup}
					\pgfsetfillopacity{0.6}
					\filldraw[fill=red, draw=red!50!black] (0,0) -- (3,-2) -- (3,0) -- cycle;
				\end{pgftransparencygroup}
	\draw (0,0) -- (-1,-3);
	\draw (0,0) -- (2,-3);
	\draw (0,0) -- (3,-2);
	\filldraw (0,0) circle (3pt);
	\draw[fill=blue] (1,-3) circle (0pt)
	node[below=0.7cm]{$\left(\frac{d-1}{d}\right)P$};
\end{tikzpicture}
\begin{tikzpicture}[scale=0.6]
    \draw[step=1cm,gray,very thin] (-0.9,-2.9) grid (2.9,0.9);
    \draw[->] (-1,0) -- (3,0);
    \draw[->] (0,1) -- (0,-3);
                \begin{pgftransparencygroup}
					\pgfsetfillopacity{0.6}
					\filldraw[fill=gray, draw=gray!50!black] (0,0) -- (0,-2) --  (2,-2) -- (2,0) -- cycle;
				\end{pgftransparencygroup}
	\filldraw (0,0) circle (3pt);
	\draw[fill=blue] (1,-2) circle (0pt)
	node[below=0.7cm]{$Q$};
\end{tikzpicture}

\begin{tikzpicture}[scale=0.6]
    \draw[step=1cm,gray,very thin] (-2.9,-8.9) grid (8.9,0.9);
    \draw[->] (-3,0) -- (9,0);
    \draw[->] (0,1) -- (0,-9);
				\begin{pgftransparencygroup}
					\pgfsetfillopacity{0.7}
					\filldraw[fill=brown, draw=brown!50!black] (-1,-1) -- (-2,-2) --  (-2,-8) -- (6,-8) -- (8,-6) -- (8,0) -- (5,0) -- (5,-4) -- (4,-5) -- (-1,-5) -- cycle;
				\end{pgftransparencygroup}
    			\begin{pgftransparencygroup}
					\pgfsetfillopacity{0.7}
					\filldraw[fill=gray, draw=blue!50!black] (0,0) -- (0,-2) --  (2,-2) -- (2,0) -- cycle;
				\end{pgftransparencygroup}							
				\begin{pgftransparencygroup}
					\pgfsetfillopacity{0.7}
					\filldraw[fill=green, draw=green] (0,0) -- (-1,-1) --  (-1,-5) -- (0,-2) -- cycle;
				\end{pgftransparencygroup}	
    			\begin{pgftransparencygroup}
					\pgfsetfillopacity{0.7}
					\filldraw[fill=blue, draw=blue] (0,-2) -- (-1,-5) --  (4,-5) -- (2,-2) -- cycle;
				\end{pgftransparencygroup}											
    			\begin{pgftransparencygroup}
					\pgfsetfillopacity{0.7}
					\filldraw[fill=yellow, draw=yellow] (2,-2) -- (4,-5) -- (5,-4) -- (2,-2) -- cycle;
				\end{pgftransparencygroup}
    			\begin{pgftransparencygroup}
					\pgfsetfillopacity{0.7}
					\filldraw[fill=red, draw=red] (2,-2) -- (5,-4) -- (5,0) -- (2,0) -- cycle;
				\end{pgftransparencygroup}	
	\filldraw (0,0) circle (3pt);
	\draw[fill=blue] (3,-8) circle (0pt)
	node[below=0.7cm]{$Q+P$ divided into $I$ and $B$};
\end{tikzpicture}
\end{center}

In the next step we will be using our first sudivision to cover the
boundary part. We will then show that covering $I$ is easy because it
lies in $0+P$.
\bigskip

\noindent
\textbf{STEP 2.1 - COVERING B:}

Let $x\in B$, then $x\not\in Q$ as $0\in \frac{d-1}{d}P$ and hence we
can find facets $F_P\prec P$ and $F_Q\prec Q$ such that $x \in F_Q +
(\conv(v,F_P))$ coming from our subdivision in STEP $1$. Hence $x$ can
be written as $x = q + \mu p$, with $q \in F_Q \prec Q$, $p\in
F_P\prec P$ and $0\leq \frac{d-1}{d} \leq \mu \leq 1$.
Then $z:= q +\frac{d-1}{d} p \in \frac{d-1}{d} F_P +F_Q$. Furthermore
$(F_Q, \frac{d-1}{d} F_P)$ is convex-normal by induction as
$\cN(F_Q)\prec \cN(\frac{d-1}{d}F_P)$ and given edges $e_{F_Q}\prec
F_Q$ and $\frac{d-1}{d}e_{F_P}\prec \frac{d-1}{d}F_P
\,(\Leftrightarrow e_{F_P}\prec F_P)$ we have
\[
\ell\left(\frac{d-1}{d}e_{F_P}\right)=\left(\frac{d-1}{d}\right)\ell\left(e_{F_P}\right)\geq
\left(\frac{d-1}{d}\right)\cdot d \ell(e_{F_Q}) = (d-1)
\ell(e_{F_Q}).
\]
Hence we can find a point $g \in G(F_Q)$ such that $z\in g +
\frac{d-1}{d} F_P$, and since $p\in (\conv(0,F_P))$ we get $x \in g +
(\conv(0,F_P))\subseteq g+ P$.\\

\noindent
\textbf{STEP 2.2 - COVERING I:}

Now we are left with covering the points in the inner part $I$ of $P +
Q $.
We claim that $I \subseteq P$, which implies $I \subseteq 0 + P
\subseteq G(Q) + P$. First we reformulate the problem by using that $I
= \left(\frac{d-1}{d}\right)P + Q\subseteq P$ is equivalent to $Q
\subseteq \frac{1}{d}P$.\\
To show the latter, suppose $Q \not \subseteq \frac{1}{d}P$, then there
exists a vertex $u$ of $Q$ that does not lie in $\frac{1}{d}P$. Hence
there exists a functional $c$ such that $c^tu=b$ and $c^tx<b$ for all
$x\in \frac{1}{d}P$. When we use the simplex method to maximize $c$
over $\frac{1}{d}P$ starting in $0$, we get a monotone edge path from
$0$ to an optimal  $u'$. As $\cN(Q)\prec \cN(\frac{1}{d}P)$ we have an
inclusion-preserving map $\cL(\frac{1}{d}P)\to \cL(Q)$ between the two
face lattices. Using this map, we get a corresponding edge path in
$Q$, which also ends in an optimal vertex $u''$, as $c\in C_{u'}
\subseteq C_{u''}$. But as edges in $\frac{1}{d}P$ are at least as
long as the corresponding (parallel) edges in $Q$, we have
\[ 
c^tu'\geq c^tu'' = c^tu \quad\lightning
\]
Hence no vertex of $Q$ is lying outside of $\frac{1}{d}P$, so that
$Q\subseteq \frac{1}{d}P$ which finishes our proof.
\end{proof}

\begin{cor}
Let $P$ be a rational polytope and $Q$ be a lattice polytope, with
\[
Q = Q_1 +\ldots + Q_s
\]
where the $Q_i$ are lattice polytopes such that the pairs $(Q_i,P)$
are convex-normal for all $i$. (For example, they could satisfy the
conditions of the previous Theorem.)\\
Then $(Q,P)$ is convex-normal.
\end{cor}

\begin{proof}
\noindent
As $(Q_i,P)$ are convex-normal we get:
\[
\begin{array}{ccc}
	Q + P 	&	=	&	(Q_1 + \ldots + Q_s) + P\\
			&	=	&	G(Q_1)+\ldots + G(Q_s) + P\\
			&\subseteq&	G(Q_1+\ldots + Q_s) + P\\
			&	=	&	G(Q)+P
\end{array}
\]
where the second equality is true because the Minkowski sum is
commutative and associative and the inclusion is true because the
$Q_i$ are lattice polytopes.
\end{proof}

\bibliographystyle{alpha}

\bibliography{ConvNormal}

\end{document}